\documentclass[11pt,draftcls,onecolumn]{IEEEtran}
\usepackage{amssymb,amsmath,color,psfrag}
\usepackage{tikz}
\usepackage{geometry}
\renewcommand{\natural}{{\mathbb{N}}}

\newcommand{\real}{{\mathbb{R}}}
\newcommand{\subscr}[2]{{#1}_{\textup{#2}}}
\newcommand{\union}{\cup}

\newcommand{\map}[3]{#1: #2 \rightarrow #3}
\newcommand{\setdef}[2]{\left\{#1 \; | \; #2\right\}}
\newcommand{\until}[1]{\{1,\dots,#1\}}

\newcommand{\const}{c}

\newcommand{\on}{\mathrm{ON}}
\newcommand{\off}{\mathrm{OFF}}

\newcommand{\lambdamax}{\lambda_{\text{max}}}

\newcommand{\stablexeq}{x_\text{eq}^1}
\newcommand{\unstablexeq}{x_\text{eq}^2}
\newcommand{\xeq}{x_{\mathrm{eq}}}
\newcommand{\xth}{x_\mathrm{th}}

\newcommand{\servicetimefunc}{\mathcal{S}}

\newcommand{\nofreetime}{T_f}
\newcommand{\controlset}{\mathcal{U}}
\newcommand{\utb}{u_{\mathrm{TP}}}

\newcommand{\rfunc}{\mathcal{R}}
\newcommand{\smin}{\mathcal{S}_{\mathrm{min}}}
\newcommand{\smax}{\mathcal{S}_{\mathrm{max}}}

\newcommand{\xmin}{x_\mathrm{min}}

\newcommand{\lambdaeq}{\subscr{\lambda}{eq}^{\textup{max}}}

\newtheorem{theorem}{Theorem}[section]

\newtheorem{lemma}[theorem]{Lemma}
\newtheorem{remark}[theorem]{Remark}

\def\QEDclosed{\mbox{\rule[0pt]{1.3ex}{1.3ex}}} 

\def\QED{\QEDclosed} 

\newcommand\oprocendsymbol{\hbox{$\square$}}
\newcommand\oprocend{\relax\ifmmode\else\unskip\hfill\fi\oprocendsymbol}

\renewcommand{\theenumi}{(\roman{enumi})}

\begin{document}

\title{Maximally Stabilizing Task Release Control Policy for a Dynamical Queue}

\author{Ketan Savla \thanks{K. Savla and E. Frazzoli are with the Massachusetts Institute of
    Technology.  \texttt{\{ksavla,frazzoli\}@mit.edu}. A preliminary version of this paper appeared in part as \cite{Savla.Nehme.ea:GNC08}. This work was supported in part by the Michigan/AFRL Collaborative Center on Control Science, AFOSR grant no. FA 8650-07-2-3744.
The authors thank Thomas Temple for helpful discussions. } \qquad Emilio Frazzoli}

\maketitle

\begin{abstract}
In this paper, we introduce a model of dynamical queue, in which the service time depends on the server utilization history.
The proposed queueing model is motivated by widely accepted empirical laws describing human performance as a function of mental arousal. 
The objective of this paper is to design task release control policies that can stabilize the queue for the maximum possible arrival rate, assuming deterministic arrivals. 
First, we prove an upper bound on the maximum possible stabilizable arrival rate for any task release control policy.
Then, we propose a simple threshold  policy that releases a task to the server only if its state is below a certain fixed value. Finally, we prove that this task release control policy ensures stability of the queue for the maximum possible arrival rate. 

\end{abstract}

\section{Introduction}

In this paper, we introduce a novel model for a dynamical queue, in which service times depend on the utilization history of the server. In other words, we consider the server as a dynamical system, and model the service time as a function of its state. Given this model, we consider the case in which new tasks arrive at a deterministic rate, and propose a task release control architecture that schedules the beginning of service of each task after its arrival. We propose a simple threshold  policy that releases a task to the server only if its state is below a certain fixed value. The proposed task release control policy is proven to ensure stability of the queue for the maximum possible arrival rate, where the queue is said to be stable if the number of tasks awaiting service does not grow unbounded over time.
  
Queueing theory is a framework to study systems with waiting lines.
 An extensive treatment of queueing systems can be found in several texts, e.g., see \cite{Kleinrock:75, Asmussen:03}. The queueing system considered in this paper falls in the category of queueing systems with state-dependent parameters, e.g., see \cite{Dshalalow:97}. In particular, we consider a queueing system with state-dependent service times.
Such systems are useful models for many practical situations, especially when the server corresponds to a human operator in a broad range of settings including, for example, human operators supervising Unmanned Aerial Vehicles, and job floor personnel in a typical production system. The model for state-dependent service times in this paper is inspired by a well known empirical law from psychology---the Yerkes-Dodson law~\cite{Yerkes.Dodson.1908}---which states that human performance increases with mental arousal up to a point and decreases thereafter. Our model in this paper is in the same spirit as the one in \cite{Bekker.Borst.06}, where the authors consider a state-dependent queueing system whose service rate is first increasing and then decreasing as a function of the amount of outstanding work. However, our model differs in the sense that the service times are related to the utilization history rather than the outstanding amount of work. A similar model has also been reported in the human factors literature, e.g., see \cite{Cummings.Nehme:09}. 

The control architecture considered in this paper falls under the category of \emph{task release control}, which has been typically used in production planning to control the release of jobs to a production system in order to deal with machine failures, input fluctuations and variations in operator workload (see, e.g.,~\cite{Glassey.Resende:88,Bertrand.VanOoijen:02}). The task release control architecture is different from an \emph{admission control} architecture, e.g., see \cite{Stidham:85,Bekker.Borst.06}, where the objective is, given a measure of the \emph{quality of service} to be optimized, to determine criteria on the basis of which to accept or reject incoming tasks. In the setting of this paper, no task is dropped and the task release controller simply acts like a switch regulating access to the server and hence effectively determines the schedule for the beginning of service of each task after its arrival. 

The contributions of the paper are threefold. First, we propose a novel dynamical queue, whose server characteristics are inspired by empirical laws relating human performance to utilization history. Second, we provide an upper bound on the maximum possible stabilizable arrival rate for any task release control policy by postulating a notion of one-task equilibrium for the dynamical queue and exploiting its optimality. Third, we propose a simple threshold policy that matches this bound, thereby also giving the stability condition for this queue. 


\section{Problem Formulation}
\label{sec:problem-formulation}
Consider the following deterministic single-server queue model. Tasks arrive periodically, at rate $\lambda$, i.e., a new task arrives every $1/\lambda$ time units. The tasks are identical and independent of each other and need to be serviced in the order of their arrival. We next state the dynamical model for the server, which determines the service times for each task.

\subsection{Server Model}
Let $x(t)$ be the server state at time $t$, and let $b: \mathbb{R} \to \{0, 1\}$ be such that $b(t)$ is 1 if the server is busy at time $t$, and 0 otherwise. 
The evolution of $x(t)$ is governed by a simple first-order model:
\begin{equation}
\dot{x}(t)=\left( b(t)-x(t) \right) / \tau,  \qquad x(0)=x_0,
\label{eq:server-dynamics}
\end{equation}
where $\tau$ is a time constant that determines the extent to which past utilization affects the current state of the server, and $x_0 \in [0,1]$ is the initial condition. Note that the set $[0,1]$ is invariant under the dynamics in Equation~\eqref{eq:server-dynamics} for any $\tau>0$ and any $b: \mathbb{R} \to \{0, 1\}$.

The service times are related to the state $x(t)$ through a map $\map{\servicetimefunc}{[0,1]}{\real_{>0}}$. If a task is allocated to the server at state $x$, then the service time rendered by the server on that task is $\servicetimefunc(x)$.
Since the controller cannot interfere the server while it is servicing a task, the only way in which it can control the server state is by scheduling the beginning of service of tasks after their arrival. Such controllers are called task release controllers and will be formally characterized later on.
In this paper we assume that:
$\servicetimefunc(x)$ is positive valued, continuous and convex. 
Let $\smin:=\min\setdef{\servicetimefunc(x)}{x \in [0,1]}$, and $\smax:=\max\{\servicetimefunc(0),\servicetimefunc(1)\}$.

The solution to Equation~\eqref{eq:server-dynamics} is 
$x(t)=e^{-t/\tau} \left(\int_0^t  \frac{1}{\tau} b(s) e^{s/\tau} ds + x_0 \right)$.
This implies that the server state $x(t)$ is increasing at times when the server is busy, i.e, when $b(t)=1$, and decreasing at times when the server is not busy, i.e., when $b(t)=0$. Note that $\servicetimefunc(x)$ is not necessarily  monotonically increasing in $x$, since it has been noted in the human factors literature (e.g., see \cite{Yerkes.Dodson.1908}) that, for certain cognitive tasks demanding persistence, the performance (which in our case would correspond to the inverse of $\servicetimefunc(x)$) could \emph{increase} with the state $x$ when $x$ is small. This is mainly because a certain minimum level of mental arousal is required for good performance. 
An experimental justification of this server model in the context of humans-in-loop systems is included in our earlier work~\cite{Savla.Nehme.ea:GNC08}, where $\servicetimefunc(x)$ for that setup was found to have a U-shaped profile.


\subsection{Task Release Control Policy}
We now describe task release control policies for the dynamical queue. Without explicitly specifying its domain, a task release controller $u$ acts like an on-off switch at the entrance of the queue. Therefore, in short, $u$ is a task release control policy if 
$u(t) \in \{\on,\off\}$ for all $t \geq 0$, and an outstanding task is assigned to the server if and only if the server is idle, i.e., when it is not servicing a task, \emph{and} when $u=\on$. Let $\controlset$ be the set of all such task release control policies. Note that we allow $\controlset$ to be quite general in the sense that it includes control policies that are functions of $\lambda$, $\servicetimefunc$, $x$, etc.

\subsection{Problem Statement}
We now formally state the problem. For a given $\tau>0$, let $n_u(t,\tau,\lambda,x_0,n_0)$ be the queue length, i.e., the number of outstanding tasks, at time $t$, under task release control policy $u \in \controlset$, when the task arrival rate is $\lambda$ and the server state and the queue length at time $t=0$ are $x_0$ and $n_0$ respectively. Define the maximum stabilizable arrival rate for policy $u$ as:
\begin{equation*}
\lambdamax(\tau,u)=\sup \setdef {\lambda}{\limsup_{t \to +\infty} n_u(t,\tau,\lambda,x_0,n_0) < +\infty,  \forall x_0 \in [0,1], \quad \forall n_0 \in \natural}.
\end{equation*}

The maximum stabilizable arrival rate over all policies is defined as
$\lambdamax^*(\tau)=\sup_{u \in \controlset} \lambdamax(\tau,u)$.
A task release control policy $u$ is called \emph{maximally stabilizing} if, for any $x_0 \in [0,1]$, $n_0 \in \natural$, $\tau>0$, $\limsup_{t \to +\infty} n_u(t,\tau,\lambda,x_0,n_0) < +\infty$ for all $\lambda \leq \lambdamax^*(\tau)$,
The objective in this paper is to design a maximally 
stabilizing task release control policy for the dynamical queue whose server state evolves according to Equation~\eqref{eq:server-dynamics}, and whose service time function $\servicetimefunc(x)$ is positive, continuous and convex.



\section{Upper Bound}
\label{sec:upper-bound}
In this section, we prove an upper bound on $\lambdamax^*(\tau)$. We do this in several steps. We start by introducing a notion of \emph{one-task equilibrium} for the dynamical queue under consideration.

\subsection{One-task Equilibrium}
Let $x_i$ be the server state at the beginning of service of the $i$-th task and let the queue length be zero at that instant. The server state upon the arrival of the $(i+1)$-th task is then obtained by integration of \eqref{eq:server-dynamics} over the time period $[0, 1/\lambda]$, with initial condition $x_0 = x_i$. Let $x_i'$ denote the server state when it has completed service of the $i$-th task. Then, $x_i'=1-(1-x_i)e^{-\servicetimefunc(x_i)/\tau}$. Assuming that  $\servicetimefunc(x_i) \leq 1/\lambda$, we get that $x_{i+1} = x_i' e^{-\left(1/\lambda - \servicetimefunc(x_i)\right)/\tau},$ and finally
$x_{i+1} =  (1-(1-x_i)e^{-\servicetimefunc(x_i)/\tau})e^{(\servicetimefunc(x_i)-1/\lambda)/\tau} 
=  \left(x_i-1+e^{\servicetimefunc(x_i)/\tau} \right)e^{-\frac{1}{\lambda\tau}}$.
If $\lambda$ and $\tau$ are such that $x_{i+1}=x_i$, then under the trivial control policy $u(t) \equiv \on$, the server state at the beginning of all the tasks after and including the $i$-th task will be $x_i$. We then say that the server is at \emph{one-task equilibrium} at $x_i$.
Therefore, for a given $\lambda$ and $\tau$, the one-task equilibrium server states correspond to $x \in [0,1]$ that satisfy 
$x = \left(x-1+e^{\servicetimefunc(x)/\tau}\right)e^{-\frac{1}{\lambda\tau}}$ and $\servicetimefunc(x) \leq 1/\lambda$,
i.e.,  
$\servicetimefunc(x) = \tau \log \left(1-(1-e^{\frac{1}{\lambda\tau}})x\right)$ and $\servicetimefunc(x) \leq 1/\lambda$. 
Let us define a map $\map{\rfunc}{[0,1] \times \real_+ \times \real_+}{\real_+}$ as:
\begin{equation}
\label{eq:r-def}
\rfunc(x,\tau,\lambda) := \tau \log \left(1-(1-e^{\frac{1}{\lambda\tau}})x\right).
\end{equation}
The following result establishes a key property of $\rfunc(x,\tau,\lambda)$.

\begin{lemma}
\label{lem:r-concave}
For any $\tau>0$ and $\lambda>0$, the function $\rfunc$ defined in Equation~\eqref{eq:r-def} is strictly concave in $x$, and $\frac{\partial}{\partial x} \rfunc(x,\tau,\lambda) > 0$ for all $x \in [0,1]$.
\end{lemma}
\begin{proof}
Taking the first and second partial derivatives of Equation~\eqref{eq:r-def} with respect to $x$, we get that,
\begin{equation*}
\frac{\partial}{\partial x} \rfunc(x,\tau,\lambda) =  \frac{-\tau\left(1-e^{\frac{1}{\lambda \tau}} \right)}{1-(1-e^{\frac{1}{\lambda\tau}})x}, \quad
\frac{\partial^2}{\partial x^2} \rfunc(x,\tau,\lambda)= \frac{-\tau (1-e^{\frac{1}{\lambda \tau}})^2}{[1-(1-e^{\frac{1}{\lambda \tau}})x]^2}.
\end{equation*}
These expressions show that, for a given $\tau>0$ and $\lambda>0$, $\frac{\partial^2}{\partial x^2} \rfunc(x,\tau,\lambda) < 0$ for all $x \in \real$. Therefore, $\rfunc$ is strictly concave in $x$. Also, $\frac{\partial}{\partial x} \rfunc(x,\tau,\lambda)|_{x=1} = \tau\left(1-e^{-\frac{1}{\lambda \tau}} \right) > 0$ for all $\tau>0$ and $\lambda>0$. Therefore, by the concavity of $\rfunc$, $\frac{\partial}{\partial x} \rfunc(x,\tau,\lambda) > 0$ for all $x \in [0,1]$.
\end{proof}
For a given $\tau>0$ and $\lambda>0$, define the set of one-task equilibrium server states as:
\begin{equation}
\label{eq:xeq-def}
\xeq(\tau,\lambda):=\setdef{x \in [0,1]}{\servicetimefunc(x)=\rfunc(x,\tau,\lambda)}.
\end{equation}

\begin{remark}
Note that we did not include the constraint $\servicetimefunc(x) \leq 1/\lambda$ in the definition of $\xeq(\tau,\lambda)$ in Equation~\eqref{eq:xeq-def}. This is because this constraint can be shown to be redundant as follows. Equation~\eqref{eq:r-def} and Lemma~\ref{lem:r-concave} imply that, for any $\tau>0$ and $\lambda>0$, $\rfunc(x,\tau,\lambda)$ is strictly increasing in $x$ and hence $\rfunc(x,\tau,\lambda) \leq \rfunc(1,\tau,\lambda)=1/\lambda$ for all $x \in [0,1]$. Therefore, $\servicetimefunc(\xeq(\tau,\lambda)) = \rfunc(\xeq(\tau,\lambda),\tau,\lambda) \leq 1/\lambda$.
\end{remark}

We introduce a couple of more definitions.
For a given $\tau>0$, let
\begin{equation}
\label{eq:lambdaeq-def}
\begin{split}
\lambdaeq(\tau):=&\max \setdef{\lambda>0}{\xeq(\tau,\lambda) \neq \emptyset}, \\ \xth(\tau):=&\xeq\left(\tau,\lambdaeq(\tau)\right).
\end{split}
\end{equation}

We now argue that the definitions in Equation~\eqref{eq:lambdaeq-def} are well posed. Consider the function $\servicetimefunc(x)-\rfunc(x,\tau,\lambda)$. Since $\rfunc(0,\tau,\lambda)=0$ for any $\tau>0$ and $\lambda>0$, and $\servicetimefunc(0)>0$, we have that $\servicetimefunc(0)-\rfunc(0,\tau,\lambda)>0$ for any $\tau>0$ and $\lambda>0$. Since $\rfunc(1,\tau,\lambda)=1/\lambda$, $\servicetimefunc(1)-\rfunc(1,\tau,\lambda)<0$ for all $\lambda < 1/\smax$. Therefore, by the continuity of $\servicetimefunc(x)-\rfunc(x,\tau,\lambda)$, the set of equilibrium server states, as defined in Equation~\eqref{eq:xeq-def}, is not-empty for all $\lambda<1/\smax$. 
Moreover, since $\rfunc(x,\tau,\lambda) \leq \rfunc(1,\tau,\lambda)=1/\lambda$ for all $x \in [0,1]$, $\servicetimefunc(x)-\rfunc(x,\tau,\lambda) \geq \servicetimefunc(x)-1/\lambda$ for all $x \in [0,1]$. Therefore, for all $\lambda > 1/\smin$, the set of equilibrium states, as defined in Equation~\eqref{eq:xeq-def}, is empty. Hence, $\lambdaeq(\tau)$ and $\xth(\tau)$ are well defined. In general, for a given $\tau>0$ and $\lambda>0$, $\xeq(x,\tau)$ is not a singleton, e.g., see Figure~\ref{fig:sfunc}. However, due to the strict convexity of $\servicetimefunc(x)-\rfunc(x,\tau,\lambda)$ in $x$ as implied by Lemma~\ref{lem:r-concave}, $\xth(\tau)$ contains only one element. In the rest of the paper, $\xth(\tau)$ will denote this single element.

\begin{figure}
\begin{center}

 \includegraphics[width=0.8\linewidth]{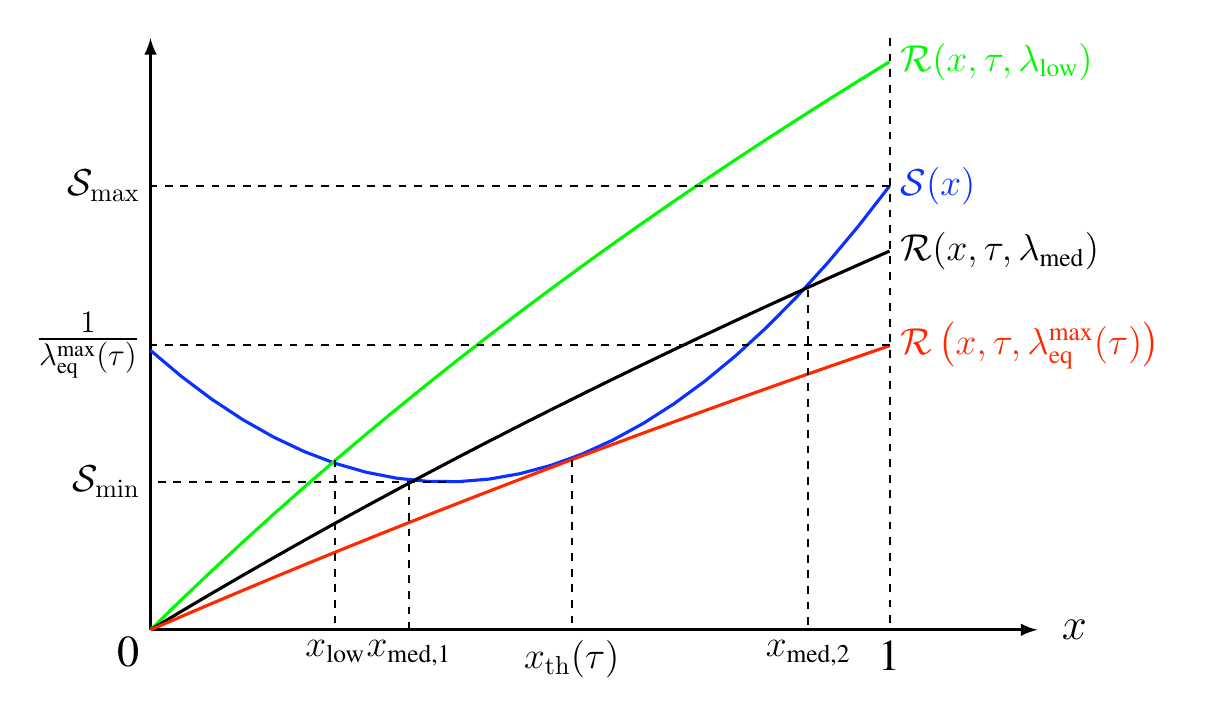}
\end{center}
\caption{A typical $\servicetimefunc(x)$ along with $\rfunc(x,\tau,\lambda)$ for three values of $\lambda$: $\lambda_{\text{low}}$, $\lambda_{\text{med}}$ and $\lambdaeq(\tau)$ in the increasing order. Also, $\xeq\left(\tau,\lambda_{\text{low}}\right)=\{x_{\text{low}}\}$, $\xeq\left(\tau,\lambda_{\text{med}}\right)=\{x_{\text{med,1}},x_{\text{med,2}}\}$ and $\xeq\left(\tau,\lambdaeq(\tau)\right)=\{\xth(\tau)\}$. Note that, since $\xth(\tau)<1$,  $\lambdaeq(\tau)$ is the value of $\lambda$ at which $\rfunc(x,\tau,\lambda)$ is tangent to $\servicetimefunc(x)$.}
\label{fig:sfunc}
\end{figure}


In the rest of the paper, we will restrict our attention on those $\tau$ and $\servicetimefunc(x)$ for which $\xth(\tau)<1$. Loosely speaking, this is satisfied when $\servicetimefunc(x)$ is increasing on some interval in $[0,1]$ and the increasing part is \emph{steep enough} (e.g., see Figure~\ref{fig:sfunc}). It is reasonable to expect this assumption to be satisfied in the context of human operators whose performance deteriorates quickly at very high utilizations. The implications of the case when $\xth(\tau)=1$ are discussed briefly at appropriate places. 

The following property of $\servicetimefunc(x)$ will be used later on.
\begin{lemma}
For any $\tau>0$, if $\xth(\tau)<1$, then $\frac{d}{dx} \servicetimefunc(x) |_{x=\xth(\tau)}>0$.
\label{lem:S-slope}
\end{lemma}
\begin{proof}
The convexity of $\servicetimefunc$ along with strict concavity of $\rfunc$ from Lemma~\ref{lem:r-concave} imply that $\servicetimefunc(x)-\rfunc(x,\tau,\lambda)$ is strictly convex in $x$ for any $\tau>0$ and $\lambda>0$. Therefore, by the definition of $\xth(\tau)$ and $\lambdaeq(\tau)$, if $\xth(\tau)<1$ then $\xth(\tau)$ corresponds to the unique minimum of $\servicetimefunc(x)-\rfunc\left(x,\tau,\lambdaeq(\tau)\right)$. 
Hence, $\frac{\partial}{\partial x} \left(\servicetimefunc(x)-\rfunc\left(x,\tau,\lambdaeq(\tau)\right)\right) |_{x=\xth(\tau)}=0$ for any $\tau>0$. The result follows by combining this with the fact that $\frac{\partial}{\partial x} \rfunc\left(x,\tau,\lambdaeq(\tau)\right)|_{x=\xth(\tau)} > 0$ for any $\tau>0$ from Lemma~\ref{lem:r-concave}.
\end{proof}

\begin{remark}
\label{rem:assumption}
The proof of Lemma~\ref{lem:S-slope} implies that, if $\xth(\tau)<1$ then $\frac{\partial}{\partial x} \left(\servicetimefunc(x)-\rfunc\left(x,\tau,\lambdaeq(\tau)\right)\right)>0$ for all $x \in (\xth(\tau),1]$ and hence $\servicetimefunc(x)-\rfunc\left(x,\tau,\lambdaeq(\tau)\right)>0$ for all $x \in (\xth(\tau),1]$. In particular, $\servicetimefunc(1)>\rfunc\left(1,\tau,\lambdaeq(\tau)\right)=\frac{1}{\lambdaeq(\tau)}$, i.e., 
$\lambdaeq(\tau)$ (which will be proven to be the maximum stabilizable arrival rate) is strictly greater than $1/\servicetimefunc(1)$, which is the rate at which the server is able to service tasks starting with the initial condition $x_0=1$ and servicing tasks continuously thereafter. 
\end{remark}
We next consider a \emph{static} problem and establish results there that will be useful for the dynamic case.
\subsection{The Static Problem}
Consider the following \emph{$n$-task static problem}: Given $n$ tasks, what is the fastest way for the dynamical server to service these tasks starting with an initial state $x$ and ending at final state $x$. We emphasize here that all the $n$ tasks are initially enqueued and no new tasks arrive.
Let $\nofreetime(x,\tau,n,u)$ be the time required by the task release control policy $u \in \mathcal{U}$ for the $n$-task static problem with initial and final server state $x \in [0,1]$. We first provide a result that relates the time for the one-task static problem to $\lambdaeq(\tau)$.

\begin{lemma}
\label{lem:lambda2-expr}
For any $x \in [0,1]$, $\tau>0$ and $u \in \controlset$, we have that
$\nofreetime(x,\tau,1,u) \geq 1/\lambdaeq(\tau)$.
\end{lemma}

\begin{proof}
First consider the policy $\tilde{u}$ that assigns the task to the server right away. In this case, $\nofreetime(x,\tau,1,\tilde{u})$ is the sum of $\servicetimefunc(x)$ and the idle time to allow the server state to return to $x$. From the definition of one-task equilibrium server states, it follows that $\nofreetime(x,\tau,1,\tilde{u})$ is such that,
\begin{equation}
x \in \xeq\left(\tau,\frac{1}{\nofreetime(x,\tau,1,\tilde{u})}\right).
\label{eq:x-eq}
\end{equation}
In other words, $\nofreetime(x,\tau,1,\tilde{u})$ is the inverse of the arrival rate such that if the server starts at state $x$ with zero queue length, then the server will be able to service a task and get back to state $x$ exactly at the instant when the next task arrives.
Now, consider a policy $u_{x^-}$ that waits for some initial time until the server state reaches state $x^-$ before assigning the task to the server. By definition, $\tilde{u}=u_x$. In this case, also referring to Figure~\ref{fig:one-task-rearrange}, $\nofreetime\left(x,\tau,1,u_{x^-}\right)$ is the sum of initial idle time $t^-$ for the server to reach state $x^-$, the service time $\servicetimefunc(x^-)$ and the idle time $t^+$ for the server state to return to $x$. Note that only those $u_{x^-}$ are feasible for which the server state after service time $\servicetimefunc(x^-)$ is not less than $x$. From the rearrangement argument illustrated in Figure~\ref{fig:one-task-rearrange}, it can be inferred that, for any such $u_{x^-}$, $\nofreetime(x,\tau,1,u_{x^-}) = \nofreetime(x^-,\tau,1,\tilde{u})$.
Therefore,  
$\max_{x \in [0,1]} \sup_{\text{feasible } u_{x^-}} \frac{1}{\nofreetime\left(x,\tau,1,u_{x^-}\right)} = \max_{x^- \in [0,1]}\frac{1}{\nofreetime(x^-,\tau,1,\tilde{u})}= \max_{x \in [0,1]}\frac{1}{\nofreetime(x,\tau,1,\tilde{u})}$,
i.e., it is sufficient to consider only the $\tilde{u}$ policy for the lemma. From Equation~\eqref{eq:x-eq}, it follows that,
\begin{equation*}
 \max_{x \in [0,1]}\frac{1}{\nofreetime(x,\tau,1,\tilde{u})}=\max \setdef{\lambda>0}{\exists x \in [0,1] \\ \text{ s.t. } x \in \xeq(\tau,\lambda)}.
\end{equation*}
The result follows from Equation~\eqref{eq:lambdaeq-def}. 
\begin{figure}[htb!]
\begin{center}

 \includegraphics[width=0.8\linewidth]{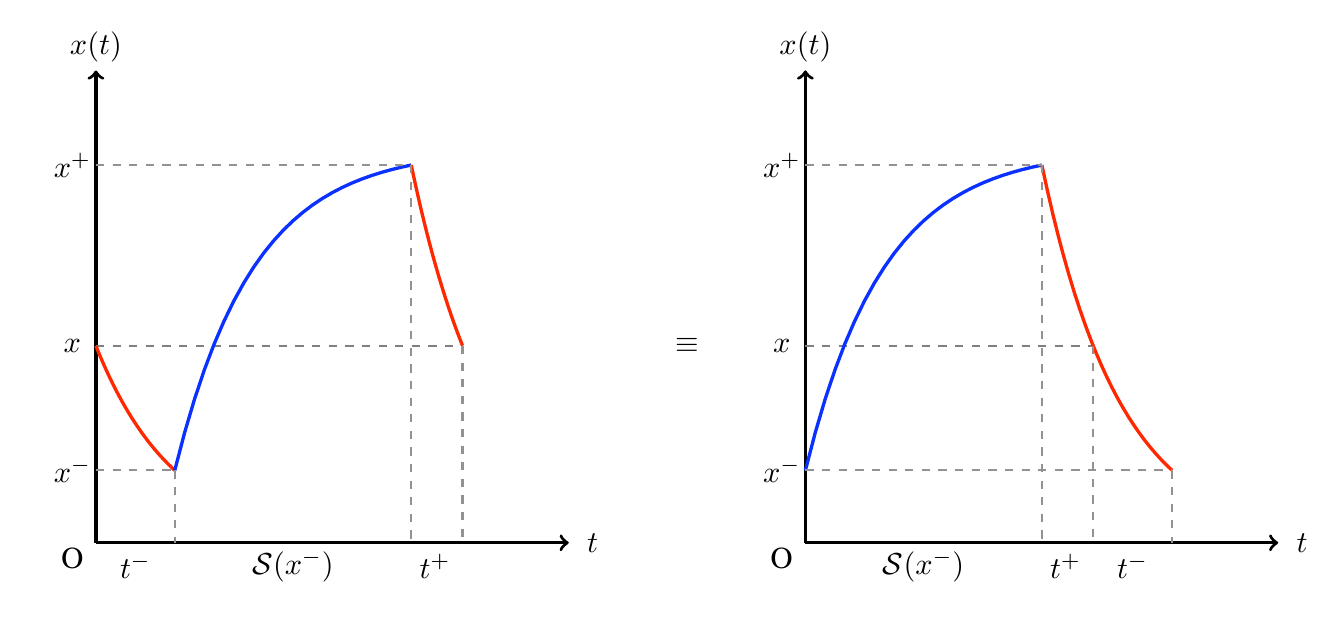}

\end{center}
\caption{Rearranging the time segments during the service cycle of a one-task static problem.}
\label{fig:one-task-rearrange}
\end{figure}
\end{proof}

The following bound on $\nofreetime (x,\tau,n,u)$ will be critical in proving a sharp upper bound on $\lambdamax^*(\tau)$.

\begin{lemma}
\label{prop:constrained-time-bound}
For any $x \in [0,1]$, $\tau>0$, $n \in \natural$ and $u \in \controlset$, we have that
$\nofreetime(x,\tau,n,u) \geq n/\lambdaeq(\tau)$.
\end{lemma}
\begin{proof}
For a given $x \in [0,1]$, $\tau>0$ and $u \in \controlset$, we prove the result by induction on $n$. The statement holds true for $n=1$ by Lemma~\ref{lem:lambda2-expr}. Assume that the result holds true for some $n=k$, i.e., $\nofreetime(x,\tau,k,u) \geq k/ \lambdaeq(\tau)$.
 Given this, we shall prove that the statement holds true for $n=k+1$. 
 
Without any loss of generality, assume that $u$ does not let the server idle before assigning the first task. This is because if $u$ lets the server idle initially until it reaches a state, say $\tilde{x}<x$, then one can alternately consider a modified $(k+1)$-task problem with initial and final server state $\tilde{x}$, and a modified control policy $u_{\text{mod}}$ that does not idle the server before assigning the first task and under which the server states at the beginning of the service of tasks are the same as those under $u$.
By following an argument similar to the one illustrated in Figure~\ref{fig:one-task-rearrange}, one can then see that the time for this modified $(k+1)$-task problem, $\nofreetime(\tilde{x},\tau,k+1,u_{\text{mod}})$, is the same as $\nofreetime(x,\tau,k+1,u)$.

For $i \in \until{k+1}$, let $x_i$ and $x_i'$ denote the server states at the beginning and at the end of service of task $i$ respectively under the policy $u$. As argued before, we assume without loss of generality that $u$ is such that $x_1=x$, and hence,
\begin{equation}
x_1' > x.
\label{eq:x_1-ineq}
\end{equation}
For $i \in \until{k+1}$, let $t_+(x_i,x_i')=\servicetimefunc(x_i)$ denote the time required for the server state to go from $x_i$ to $x_i'$ and for $i \in \until{k}$, let $t_-(x_i',x_{i+1}) = \tau (\log x_i' - \log x_{i+1})$ denote the time required by the server to idle from $x_i'$ to $x_{i+1}$. With these notations, one can write that,
 \begin{equation}
 \label{eq:split}
\nofreetime(x,\tau,k+1,u) =  \sum_{i=1}^{k} t_+(x_i,x_i') + \sum_{i=1}^{k} t_-(x_i',x_{i+1}) + t_+(x_{k+1},x_{k+1}')+ t_-(x_{k+1}',x).
 \end{equation}
The rest of the proof is split among the following two cases.\\
Case 1: $x_{k+1} \geq x$. We write $t_-(x_{k+1}',x)$ as
\begin{equation}
t_-(x_{k+1}',x) = \tau \log (x_{k+1}'/  x_{k+1})+ \tau \log (x_{k+1} / x).
\label{eq:t-last}
\end{equation}
Therefore, from Equations~\eqref{eq:split} and \eqref{eq:t-last}, we get that,
\begin{align*}
\nofreetime(x,\tau,k+1,u)
    = &  \sum_{i=1}^{k} t_+(x_i,x_i') + \sum_{i=1}^{k-1} t_-(x_i',x_{i+1}) 
 +  t_+(x_{k+1},x_{k+1}') +  \tau \log x_{k+1}' \\ & - \tau \log x_{k+1})   + t_-(x_{k}',x_{k+1})  +  \tau (\log x_{k+1} - \log x) 
\end{align*}
Since $ t_-(x_{k}',x_{k+1}) +  \tau (\log x_{k+1} - \log x) = \tau (\log x_k' - \log x_{k+1})+ \tau (\log x_{k+1} - \log x)=\tau(\log x_k' - \log x)=t_-(x_k',x)$, we can write that,
\begin{equation}
\label{eq:Tf-intermediate}
\begin{split}
\nofreetime(x,\tau,k+1,u)
    = & \sum_{i=1}^{k} t_+(x_i,x_i') + \sum_{i=1}^{k-1} t_-(x_i',x_{i+1}) 
 +  t_+(x_{k+1},x_{k+1}') \\ & +  t_-(x_{k+1}',x_{k+1})  + t_-(x_k',x).
\end{split}
\end{equation}
Now, consider the $k$-task static problem with initial and final server state $x$. Let $u'$ denote the control policy under which the server states at the beginning of the tasks are $x_1,x_2,\ldots,x_k$. Also, $t_+(x_{k+1},x_{k+1}') +  t_-(x_{k+1}',x_{k+1})=\nofreetime(x_{k+1},\tau,1,\tilde{u})$, where $\tilde{u}$, as in the proof of Lemma~\ref{lem:lambda2-expr}, is the control policy for the one-task static problem that assigns the task to the server without any initial idling time. Combining these with Equation~\eqref{eq:Tf-intermediate}, one gets that,
$\nofreetime(x,\tau,k+1,u) = \nofreetime(x,\tau,k,u') + \nofreetime(x_{k+1},\tau,1,\tilde{u})$.
We have $\nofreetime(x,\tau,k,u') \geq k/\lambdaeq(\tau)$ by the induction argument and $ \nofreetime(x_{k+1},\tau,1,\tilde{u}) \geq 1/\lambdaeq(\tau)$ by Lemma~\ref{lem:lambda2-expr}. Therefore, $\nofreetime(x,\tau,k+1,u) \geq (k+1)/\lambdaeq(\tau)$. \\
Case 2: $x_{k+1} < x$. Let $\tilde{i}:=\max \setdef{i \in \until{k}}{x_i' > x \quad \& \quad x_{i+1} < x}$. Since $x_1' > x$ by Equation~\eqref{eq:x_1-ineq}, $\tilde{i}$ is well defined. Therefore,
\begin{equation}
\begin{split}
\nofreetime(x,\tau,k+1,u)= & \sum_{i=1}^{\tilde{i}} t_+(x_i,x_i') + \sum_{i=1}^{\tilde{i}-1} t_-(x_i',x_{i+1})  + t_-(x_{\tilde{i}}',x_{\tilde{i}+1}) 
 \\ & + \sum_{i=\tilde{i}+1}^{k+1} t_+(x_i,x_i')   +  \sum_{i=\tilde{i}+1}^{k} t_-(x_i',x_{i+1}) + t_-(x_{k+1}',x).
\end{split}
\label{eq:Tf-case2}
\end{equation}
Splitting $t_-(x_{\tilde{i}}',x_{\tilde{i}+1})$ as $\tau(\log x_{\tilde{i}}'- \log x)+\tau(\log x - \log x_{\tilde{i}+1})$, Equation~\eqref{eq:Tf-case2} can be written as
\begin{equation}
\nofreetime(x,\tau,k+1,u)=\nofreetime(x,\tau,\tilde{i},u_1)+\nofreetime(x,\tau,k+1-\tilde{i},u_2),
\label{eq:Tf-case2-final}
\end{equation}
where $u_1$ is the control policy for the $\tilde{i}$-task static problem with initial and final server state $x$, such that the server states at the beginning of the service of tasks are $x_1,\ldots,x_{\tilde{i}}$, and $u_2$ is the control policy for the $k+1-\tilde{i}$-task static problem with initial and final server state $x$ such that the server states at the beginning of the service of the tasks are $x_{\tilde{i}+1},\ldots,x_{k+1}$. Since both $\tilde{i}$ and $k+1-\tilde{i}$ are strictly less than $k+1$, we apply induction argument to both the terms on the right side of Equation~\eqref{eq:Tf-case2-final} to conclude that
$\nofreetime(x,\tau,k+1,u) \geq \frac{\tilde{i}}{\lambdaeq(\tau)} + \frac{k+1-\tilde{i}}{\lambdaeq(\tau)} = \frac{k+1}{\lambdaeq(\tau)}$.
\end{proof}

\subsection{Upper Bound on Stabilizable Arrival Rate}
We now return to the dynamic problem, where we prove an upper bound on $\lambdamax^*(\tau)$. Trivially, $\lambdamax^*(\tau) \leq \frac{1}{\smin}$. We next establish a sharper upper bound. First, we state a useful lemma.

\begin{lemma}
For any $\tau>0$, $x_0 \in [0,1]$, $n_0 \in \natural$ and $\lambda > \lambdaeq(\tau)$, if $\xth(\tau) < 1$ then there exist constants $x_L(\tau)$ and $x_U(\tau)$ satisfying $0<x_L(\tau)<x_U(\tau)<1$ such that for any $u \in \controlset$ under which the server states at the beginning of tasks do not lie in $[x_L(\tau),x_U(\tau)]$ infinitely often, we have that
$\limsup_{t \to + \infty} n_u(t,\tau,\lambda,x_0,n_0) = +\infty$.
\label{lem:x-l-u}
\end{lemma}
\begin{proof}
We first define the constants $x_L(\tau)$ and $x_U(\tau)$. For a given $\tau>0$, let $\xmin:=1-e^{-\smin/\tau}$ denote a lower bound on the lowest possible server state immediately after the service of a task.
Note that, for any $\tau>0$ and $\smin>0$, $\xmin>0$.
For a given $\tau>0$, define a map $\map{g}{[0,1]}{\real \union \{+\infty\}}$ as:
\begin{equation}
\label{eq:g-def}
g(x)=\smin+\tau \log (\xmin / x).
\end{equation}
Note that $g$ is continuous, strictly decreasing with respect to $x$, and that $g(0)=+\infty$. Therefore, by continuity argument, there exists a $\tilde{x}>0$ such that $g(x) > 1/\lambdaeq(\tau)$ for all $x \in [0,\tilde{x})$. Define $x_{l_1}(\tau):=\min\{\xmin,\tilde{x}\}$. It follows from the previous arguments that $x_{l_1}(\tau)>0$. Define the following quantities
\begin{equation}
\label{eq:xu-def}
\begin{split}
x_{u_1}:= & \max \setdef{x \in [0,1]}{\servicetimefunc(x)=1/\lambdaeq(\tau)}, \\ x_{u_2}:= & 1-(1-x_{l_1} )e^{-\frac{2}{\tau \lambdaeq(\tau)}}, \quad x_{l_2}:=x_{u_2} e^{-\frac{2}{\tau \lambdaeq(\tau)}},
\end{split}
\end{equation}
\begin{equation*}
x_L(\tau):=\min\{x_{l_1},x_{l_2}\}, \quad x_U(\tau):=\max\left\{\frac{1+x_{u_1}}{2},x_{u_2}\right\},
\end{equation*}
where we have dropped the dependency of $x_{l_1}$, $x_{l1_2}$, $x_{u_1}$ and $x_{u_2}$ on $\tau$ for the sake of conciseness.
We now infer relevant properties of the various quantities defined in Equation~\eqref{eq:xu-def}. Remark~\ref{rem:assumption} implies that, if $\xth(\tau)<1$, then $\servicetimefunc(1)>1/\lambdaeq(\tau)$. This, combined with the fact that 
$\smin \leq 1/\lambdaeq(\tau)$ and that $\servicetimefunc(x)$ is continuous, implies that $x_{u_1}$ is well-defined and $x_{u_1}<1$.
From the definition of $x_{u_2}$, we have that $x_{u_2}<1$ and also $x_{u_2}>x_{l_1}$ since $e^{-\frac{2}{\tau \lambdaeq(\tau)}}<1$.
This, combined with earlier conclusion after Equation~\eqref{eq:g-def} that $x_{l_1}>0$, implies that $x_{u_2}>0$. Therefore, from the definition of $x_{l_2}$, we have that, $x_{l_2}>0$ and $x_{l_2}<x_{u_2}$. Combining these facts with the definitions of $x_L(\tau)$ and $x_U(\tau)$, we have that, $0<x_L(\tau)<x_U(\tau)<1$.

For the rest of the proof, we drop the dependency of $x_L$ and $x_U$ on $\tau$. Consider a $u$ such that the maximum task index for which the server state lies in $[x_L,x_U]$ is finite, say $I$. Let $x_i$ and $x_i'$ be the server states at the beginning of service of task $i$ and the end of service of task $i$ respectively. Consider a service cycle of a typical task for $i>I$. We now consider four cases depending on where $x_i$ and $x_{i+1}$ belong, and in each case we show that the time between the beginning of successive tasks is strictly greater than $1/\lambdaeq(\tau)$, thereby establishing the lemma.
\begin{itemize}
\item $x_i \in [0,x_L)$ and $x_{i+1} \in [0,x_L)$:
The service time for task $i$, $\servicetimefunc(x_i)$, is lower bounded by $\smin$. By the definition of $\xmin$, $x_i' \geq \xmin$ and hence $x_i' \geq x_L$. Since $x_{i+1}$ is less than $x_L$, the server has to idle for time $\tau \log (x_i' /x_{i+1})$, which is lower bounded by $\tau \log( \xmin/ x_L)$. In summary, the total time between the service of successive tasks is lower bounded by $\smin+\tau \log (\xmin / x_L)$, which is equal to $g(x_L)$ from Equation~\eqref{eq:g-def}. By the choice of $x_L$, $g(x_L)$ is strictly greater than $1/\lambdaeq(\tau)$. 

\item $x_i \in (x_U,1]$ and $x_{i+1} \in (x_U,1]$:
The convexity of $\servicetimefunc(x)$ along with Lemma~\ref{lem:S-slope} imply that $\servicetimefunc(x) > 1/\lambdaeq(\tau)$ for all $x \in (x_{u_1},1]$. Since $x_U>x_{u_1}$ from Equation~\eqref{eq:xu-def}, we have that $\servicetimefunc(x_i) > 1/\lambdaeq(\tau)$. Therefore, the time spent between successive tasks is lower bounded by $1/\lambdaeq(\tau)$.

\item $x_i \in [0,x_L)$ and $x_{i+1} \in (x_U,1]$:
The fact that it takes at least $2/\lambdaeq(\tau)$ amount of service time on task $i$ for the server to go from from $x_i$ to $x_{i+1}$ follows from the definition of $x_{l_1}$ and $x_{u_2}$ and their relation to $x_L$ and $x_U$ respectively, as stated in Equation~\eqref{eq:xu-def}. Therefore, the time spent between successive tasks is at least $2/\lambdaeq(\tau)$.

\item $x_i \in (x_U,1]$ and $x_{i+1} \in [0,x_L)$:
The fact that it takes at least $2/\lambdaeq(\tau)$ time for the server to idle from $x_i'$ to $x_{i+1}$ follows from the definition of $x_{l_2}$ and $x_{u_2}$ and their relation to $x_L$ and $x_U$ respectively, as stated in Equation~\eqref{eq:xu-def}. Therefore, the time spent between successive tasks is at least $2/\lambdaeq(\tau)$.
\end{itemize}
\end{proof}

\begin{theorem}
\label{thm:lambda-upperbound}
For any $\tau>0$, $x_0 \in [0,1]$, $n_0 \in \natural$, $\lambda > \lambdaeq(\tau)$ and $u \in \controlset$, if $\xth(\tau)<1$ then we have that 
$\limsup_{t \to +\infty} n_u(t,\tau,\lambda,x_0,n_0) = + \infty$.
\end{theorem}
\begin{proof}
Lemma~\ref{lem:x-l-u} implies that there exist $x_L>0$ and $x_U<1$ such that it suffices to consider set of task release control policies under which the server states at the beginning of service of tasks lie in $[x_L,x_U]$ infinitely often. Consider one such control policy and let the sequence of indices of tasks for which the server state at the beginning of their service belongs to $[x_L,x_U]$ be denoted as $i_1, i_2,\ldots$. Let $x_{i}$ and $t_{i}$ be the server state and the time respectively at the beginning of the service of the $i$-th task. We have that $x_{i_k} \in [x_L,x_U]$ for all $k \geq 1$.
Define constants $\const_1$ and $\const_2$ as follows:
\begin{equation}
\const_1:= - \tau \log x_L, \quad \const_2:=-\tau \log(1-x_U).
\label{eq:kappa-def}
\end{equation}

Note that both $\const_1$ and $\const_2$ are positive.
For each $k>1$, we now relate $t_{i_k}-t_{i_1}$ to the time for a related static problem. If $x_{i_k} \geq x_{i_1}$, then consider the $(i_k-i_1)$-task static problem with initial and final server state $x_{i_1}$. Then, for a control policy $u'$ for this static problem under which the server states are $x_{i_1}, x_{i_1+1}, \ldots,x_{i_k}$, we have $\nofreetime(x_{i_1},\tau,i_k-i_1,u')=t_{i_k}-t_{i_1}+\tau(\log x_{i_k}- \log x_{i_1})$. 
Therefore, $t_{i_k}-t_{i_1} = \nofreetime(x_{i_1},\tau,i_k-i_1,u') - \tau(\log x_{i_k}- \log x_{i_1}) \geq \nofreetime(x_{i_1},\tau,i_k-i_1,u') +\tau \log x_{i_1} \geq \nofreetime(x_{i_1},\tau,i_k-i_1,u') - \const_1$, where the last inequality follows from Equation~\eqref{eq:kappa-def}.
If $x_{i_k}<x_{i_1}$, then consider the $(i_k-i_1+m)$-task static problem with initial and final server state $x_{i_1}$ and a control policy $u''$ for this static problem such that: the server states at the beginning of the service of first $i_k-i_1$ tasks are $x_{i_1}, x_{i_1+1}, \ldots,x_{i_k}$ and $m$ is the smallest number such that on servicing these $m$ tasks without any idling after $i_k$-th task, one has $x_{i_k+m} \geq x_{i_1}$. An upper bound on the time for the static problem under $u''$ is 
\begin{equation}
\label{eq:upp-bound}
\begin{split}
\nofreetime(x_{i_1},\tau,m+i_k-i_1,u'') \leq & t_{i_k}-t_{i_1}+\tau \log (1-x_{i_1})  - \tau \log (1-x_{i_k}) \\ & +\smax  -\tau \log x_{i_1}, 
\end{split}
\end{equation}
where $t_{i_k}-t_{i_1}$ is the time between the beginning of service of tasks $i_1$ and $i_k$, $\tau \left(\log (1-x_{i_1})-\log(1-x_{i_k}) \right)$ is the time required for the server state to increase from $x_{i_k}$ to $x_{i_1}$ under continuous usage, $\smax$ is the upper bound on the \emph{overshoot} time that could happen if the server is in middle of servicing a task while crossing state $x_{i_1}$ and $-\tau \log x_{i_1}$ is the upper bound on the time taken by the server to idle back to state $x_{i_1}$ in case of the overshoot.
Equation~\eqref{eq:upp-bound} can be rewritten as $t_{i_k}-t_{i_1} \geq  \nofreetime(x_{i_1},\tau,m+i_k-i_1,u'') -\smax + \tau \log (1-x_{i_k}) + \tau \log x_{i_1} \geq \nofreetime(x_{i_1},\tau,m+i_k-i_1,u'') - \smax -\const_1 - \const_2$, where the last inequality follows from Equation~\eqref{eq:kappa-def}.

Combining these bounds on $t_{i_k}-t_{i_1}$ with lemma~\ref{prop:constrained-time-bound}, we have that, for all $k \geq 1$,
\begin{equation}
t_{i_k}-t_{i_1} \geq \left\{\begin{array}{ll} 
  \frac{i_k-i_1}{\lambdaeq(\tau)} - \const_1 & \mbox{if } x_{i_k} \geq x_{i_1},\\
   \frac{m+i_k-i_1}{\lambdaeq(\tau)} - \smax - \const_1 - \const_2 & \mbox{otherwise}.
   \end{array}  \right.
   \label{eq:t-bound}
\end{equation}

With $\const=\const_1 + \const_2 + \smax$, Equation~\eqref{eq:t-bound} becomes
\begin{equation}
\label{eq:t-ineq}
t_{i_k}-t_{i_1} \geq  \frac{i_k-i_1}{\lambdaeq(\tau)} - \const \quad \forall k \geq 1.
\end{equation}
For $k \geq 1$, let $n_k$ be the queue length at the beginning of service of task $i_k$. Then one can write
$n_k \geq n_1+ \lambda (t_{i_k}-t_{i_1}) - (i_k-i_1) \quad \forall k \geq 1$. This with Equation~\eqref{eq:t-ineq} gives
\begin{equation}
\label{eq:t-ineq2}
n_k \geq n_1-\lambda \const + (i_k-i_1) \left(\frac{\lambda}{\lambdaeq(\tau)}-1 \right) \quad \forall k \geq 1.
\end{equation}
From Equation~\eqref{eq:t-ineq2}, we get that, for $\lambda>\lambdaeq(\tau)$, $\lim_{k \to +\infty} n_k = +\infty$. The theorem follows from the fact that $\limsup_{t \to +\infty} n_u(t,\tau,\lambda,x_0,n_0) \geq \lim_{k \to +\infty} n_k$.
\end{proof}

\begin{remark}
\begin{enumerate}
\item Theorem~\ref{thm:lambda-upperbound} implies that, for a given $\tau>0$, if $\xth(\tau)<1$ then $\lambdamax^*(\tau) = \lambdaeq(\tau)$. 
\item If $\xth(\tau)=1$, then one can show that, for any $\epsilon>0$, there exists no stabilizing task release control policy for arrival rates greater than $\lambdaeq(\tau)+ \epsilon$, i.e., $\lambdamax^*(\tau) \leq \lambdaeq(\tau)+ \epsilon$.
\end{enumerate}
\end{remark}

In the next section, we propose a simple task release control policy and prove that it is maximally stabilizable, i.e., for any $\lambda \leq \lambdaeq(\tau)$, it ensures that the dynamical queue is stable.

\section{Control Policy and Lower Bound on the Stabilizable Arrival Rate}
\label{sec:control}
Consider the following threshold policy:
\begin{equation*}
\utb(t) =  \left\{\begin{array}{ll} 
  \on  & \mbox{if } x(t) \leq \xth(\tau),\\
   \off & \mbox{otherwise},
\end{array}  \right.
\end{equation*}
where $\xth(\tau)$ is defined in Equation~\eqref{eq:lambdaeq-def}. 
We now prove that this threshold policy is maximally stabilizing. 

\begin{theorem}
\label{thm:lambda-lowerbound}
For any $\tau>0$, $x_0 \in [0,1]$, $n_0 \in \natural$ and $\lambda \leq \lambdaeq(\tau)$, if $\xth(\tau)<1$ then we have that
$\limsup_{t \to +\infty} n_{\utb} (t,\tau,\lambda,x_0,n_0) < +\infty$.
\end{theorem}
\begin{proof}
Let $x_i$ and $t_i$ be the server state and time instants respectively at the beginning of service of the $i$-th task. For brevity in notation, let $n(t)$ be the queue length at time $t$. For any $x_0 \in [0,1]$ and $n_0 \in \natural$, considering the possibility when $x_0>\xth(\tau)$ we have that $n(t_1)=\max\{0,n_0-1,n_0-1+\lfloor \lambda \tau \log (x_0/\xth) \rfloor\}$.
We now prove that $n(t_i) \leq n(t_1) +\lceil \left(\lambda -1/\smax \right) \left( -\tau \log (1-\xth) + \smax\right) \rceil+ \lceil -\lambda \tau \log \xth \rceil$ for all $i$ through the following two cases:
\begin{itemize}
\item \textbf{State 1:} $x_1 = \xth$. While $n(t_i)>0$, we have that $x_{i+1}=\xth$ and $t_{i+1}-t_i=\nofreetime(\xth,\tau,1,\utb)=1/\lambdaeq(\tau)$. Therefore, if $\lambda = \lambdaeq(\tau)$, then the arrival rate is same as the service rate and hence $n(t_i) \equiv n(t_1)$ for all $i$. If $\lambda < \lambdaeq(\tau)$, then the service rate is greater than the arrival rate and hence there exists an $i' \geq 1$ such that $n(t_{i}) < n(t_{i-1})$ for all $i \leq i'$ and $n\left(t_{i'}+1/\lambdaeq(\tau)\right)=0$ and hence $x_{i'+1}<\xth$. Thereafter, we appeal to the next case by resetting $x_{i'+1}$ and $t_{i'+1}$ as $x_1$ and $t_1$ respectively. Moreover, with these notations, $n(t_1)=0$.
\item \textbf{State 2:} $x_1<\xth$. While the queue length is non-zero, the server is never idle. The maximum amount of continuous service time required for the server state to cross $\xth$ starting from any $x_1 <\xth$ is upper bounded by $-\tau \log (1-\xth) + \smax$, where $-\tau \log (1-\xth)$ is the time to go from $x=0$ to $x=\xth$ when the server is continuously busy, and the $\smax$ term accounts for the fact that the server might be in middle of servicing a task when it reaches $x=\xth$ and hence its server state will exceed $\xth$ before it finishes the task. The maximum amount of time spent in such an \emph{overshoot} is upper bounded by $\smax$. Since $1/\smax$ is the minimum rate at which the server will be servicing the tasks during this time, the increase in the queue length during this time is upper bounded by $\lceil \left(\lambda -1/\smax \right) \left( -\tau \log (1-\xth) + \smax\right) \rceil$. Under the threshold policy, the possible overshoot above $\xth$ will be followed by an idle time which is upper bounded by $-\tau \log \xth$, at the end of which the server state is $\xth$. The increase in queue length during this time  is upper bounded by $ \lceil -\lambda \tau \log \xth \rceil$. Therefore, the maximum queue length when the server state reaches $\xth$ is upper bounded by $n_1+\lceil \left(\lambda -1/\smax \right) \left( -\tau \log (1-\xth) + \smax\right) \rceil+ \lceil -\lambda \tau \log \xth \rceil$. Thereafter, we appeal to the earlier case by resetting $x_1=\xth$ and $n_1$ to be the number of outstanding tasks when the server state reaches $\xth$.
\end{itemize}
In summary, when the system is in State 1, if $\lambda= \lambdaeq(\tau)$, it stays there with constant queue length, else, the queue length monotonically decreases to zero at which point it enters State 2. When the system is in State 2, it stays in it for ever or eventually  enters State 1 with bounded queue length. Collecting these facts, we arrive at the result.
\end{proof}

\begin{remark}
\label{rem:control}
\begin{enumerate}
\item From Theorem~\ref{thm:lambda-upperbound} and Theorem~\ref{thm:lambda-lowerbound}, one can deduce that, for any $\tau>0$, $\lambdamax^*(\tau)=\lambdaeq(\tau)$, and the threshold policy is a maximally stabilizing task release control policy.
\item In general, for a given $\lambda' \leq \lambdaeq(\tau)$, the threshold policy with the threshold value set at any value in $[\stablexeq(\tau,\lambda'),\unstablexeq(\tau,\lambda')]$ would ensure stability of the queue for all values of $\lambda \leq \lambda'$. 
\item If $\xth(\tau)=1$, then one can show that, given $\epsilon>0$, there exists a $\delta(\epsilon)>0$ such that the threshold policy with the threshold value set at $1-\delta(\epsilon)$ ensures stability of the queue for all arrival rates less than or equal to $\lambdaeq(\tau)-\epsilon$.

\end{enumerate}
\end{remark}

\section{Conclusions}
\label{sec:conclusions}
In this paper, we studied the stability problem of a dynamical queue whose service times are dependent on the state of a simple underlying dynamical system. The model for the service times is loosely inspired by the performance of a human operator in a persistent mission. We proposed a simple task release control policy for such a dynamical queue and proved that it ensures stability of the queue for the maximum possible arrival rate. In future, we plan to extend the analysis here to stochastic inter-arrival and service times and to general server dynamics. We also plan to design control policies for such queues that optimize other qualities of service such as average waiting time of an incoming task using, for example, the flexibility in choosing the threshold, as noted in part (ii) of Remark~\ref{rem:control}.  


{\small
  \bibliographystyle{ieeetr}%
  \bibliography{alias,efmain,frazzoli}
}

\end{document}